\documentclass[10pt, reqno]{amsart}
\usepackage{lipsum}
\usepackage{a4wide}
\usepackage{amssymb}
\usepackage{amsmath}
\usepackage{amsthm}
\usepackage{tikz}
\usepackage{amscd}
\usepackage{hyperref}
\usepackage[all]{xy}
\usepackage{enumitem}
\hypersetup{colorlinks,linkcolor={blue},citecolor={blue},urlcolor={red}}
\usepackage{color}
\theoremstyle{plain}
\usepackage[margin=2.4cm]{geometry}
\numberwithin{equation}{section}

\newtheorem{theorem}{Theorem}[section]
\newtheorem{corollary}[theorem]{Corollary}
\newtheorem{lemma}[theorem]{Lemma}

\newtheorem{proposition}[theorem]{Proposition}

\title{On Tate cohomology and base change of representations of $D^\times$}
\author{Sabyasachi Dhar}
\begin{document}
	\maketitle
\begin{abstract}
Let $F$ be a non-Archimedean local field with residue characteristic $p$, and let $D$ be a central $F$-division algebra of degree $d$. Let $E/F$ be a finite Galois extension of prime degree $\ell$, where $\ell \ne p$ and $\ell$ does not divide $d$. Set $D_E=D\otimes_F E$. Let $\mathcal{K}$ be the maximal unramified extension of $\mathbb{Q}_\ell$ in $\overline{\mathbb{Q}}_\ell$. In this article, we explicitly compute the Tate cohomology groups of absolutely irreducible, integral, depth-zero, $\mathcal{K}$-representations of $D_E^\times$ in the context of local base change lifting, which verifies a conjecture of Treumann--Venkatesh on mod-$\ell$ functoriality.  
\end{abstract}
\section{Introduction}
Let $F$ be a non-Archimedean local field with residue characteristic $p$, and let $D$ be a central $F$-division algebra of index $n$. Let $\ell$ be a prime distinct from $p$.
The Jacquet-Langlands correspondence provides an injective map
$$ {\rm JL}_F : {\rm Irr}_{\overline{\mathbb{Q}}_\ell}(D^\times)\rightarrow {\rm Irr}_{\overline{\mathbb{Q}}_\ell}({\rm GL}_n(F)) $$
where ${\rm Irr}_{\overline{\mathbb{Q}}_\ell}(G)$, for a locally profinite group $G$, denotes the set of isomorphism classes of irreducible smooth $\ell$-adic representations of $G$. Let $E/F$ be a Galois extension of degree $\ell$. The theory of base change lifting provides a map
$$ {\rm BC}_{E/F} : {\rm Irr}_{\overline{\mathbb{Q}}_\ell}({\rm GL}_n(F))\rightarrow {\rm Irr}_{\overline{\mathbb{Q}}_\ell}({\rm GL}_n(E)). $$
Let $D_E$ be the central simple $E$-algebra $D\otimes_F E$. Assume that $\ell$ does not divide $n$. Then $D_E$ is a division algebra over $E$, and we obtain a base change lifting map
$$ {\rm bc}_{E/F}:{\rm Irr}_{\overline{\mathbb{Q}}_\ell}(D^\times)\rightarrow {\rm Irr}_{\overline{\mathbb{Q}}_\ell}(D_E^\times)$$
induced by ${\rm BC}_{E/F}$ and the Jacquet-Langlands correspondence. For an irreducible smooth depth-zero $\ell$-adic representation $\pi$ of $D^\times$, the representation $\Pi={\rm bc}_{E/F}(\pi)$ is of depth-zero and is stable under the action of ${\rm Gal}(E/F)$. In this article, we explicitly compute the Tate cohomology groups $\widehat{\rm H}^i({\rm Gal}(E/F),\mathcal{L})$ of certain $D_E^\times\rtimes {\rm Gal}(E/F)$-stable lattice $\mathcal{L}$ in $\Pi$, using the compact induction model of $\Pi$ in terms of admissible tame pairs.

The purpose of this paper is motivated by a conjecture of Treumann--Venkatesh on mod-$\ell$ functoriality (\cite[Conjecture 6.3]{TV_paper}), which proposed that one of the Jordan-Hölder factors of the $\overline{\mathbb{F}}_\ell$-representations $\widehat{\rm H}^i({\rm Gal}(E/F),\mathcal{L})$ is the mod-$\ell$ reduction of $\pi$ twisted by the Frobenius automorphism of the coefficient field $\overline{\mathbb{F}}_\ell$. During several years, there has been considerable progress towards Treumann--Venkatesh's conjecture. After an initial evidence in the context of local base change lifting for depth-zero cuspidal representations of ${\rm GL}_n(F)$ with $\ell\nmid n$ due to Ronchetti (\cite{Ronchetti_BC}), Feng (\cite{feng2023modular}, \cite{feng2024smith}) made remarkable progress towards these conjectures--using the recent advances due to the works of V. Lafforgue and Fargues-Scholze. However, the
explicit computation of Tate cohomology groups, in particular, the non-vanishing of Tate cohomology and genericity, are not well
understood. Some of these questions were answered in \cite{nadimpalli2024tate} for ${\rm GL}_n(F)$ and in \cite{dhar2025jacquet} for $F$-rational points of general connected reductive groups defined over $F$, where Treumann--Venkatesh's conjecture was proved for integral generic representations of ${\rm GL}_n(F)$ in the context of local base change lifting. The main result of this article proves the conjecture for irreducible depth-zero integral representations $D^\times$, which shows that Tate cohomology realizes the base change lifting ${\rm bc}_{E/F}$.

We now introduce some notation to state the main result. Let $\overline{\mathbb{Q}}_\ell$ be an algebraic closure of $\mathbb{Q}_\ell$. Let $\mathcal{K}$ be the maximal unramified extension of $\mathbb{Q}_\ell$ in $\overline{\mathbb{Q}}_\ell$. Let $\pi$ be an absolutely irreducible smooth depth-zero $\mathcal{K}$-representation of $D^\times$. Such $\pi$ is parametrized by a tame admissible pair $(K/F,\theta)$ (\cite[Subsection 1.5]{BH_depth_zero}), where $K$ is a finite unramified extension of $F$ and $\theta$ is a smooth character of $K^\times$, trivial on the first congruence subgroup $U_K^1$ of $K^\times$. Let $\Pi$ be the absolutely irreducible smooth depth-zero $\mathcal{K}$-representation of $D_E^\times$, parametrized by the tame admissible pair $(EK/E,\widetilde{\theta})$, where the character $\widetilde{\theta}$ is the composition of $\theta$ and the norm map ${\rm Nr}:(EK)^\times\rightarrow K^\times$. The representations $\pi$ and $\Pi$ are given by the compactly induced representations ${\rm ind}_{J}^{D^\times}(\lambda)$ and ${\rm ind}_{J_E}^{D_E^\times}(\Lambda)$ respectively, where $J$ (resp. $J_E$) is an open compact mod-center subgroup of $D^\times$ (resp. $D_E^\times$) and $\lambda$ (resp. $\Lambda$) is a character of $J$ (resp. $J_E$). Then $\Pi = {\rm bc}_{E/F}(\pi)$, and it extends uniquely to a representation of $D_E^\times\rtimes {\rm Gal}(E/F)$. Assume that $\pi$ is integral, in which case the base change lift $\Pi$ is also integral. Let $\mathcal{L}_\Pi$ be the induced lattice consisting of all functions in $\Pi$, taking values in the ring of integers $\mathcal{O}_{\mathcal{K}}$ of $\mathcal{K}$. Then $\mathcal{L}_\Pi$ is stable under the action of $D_E^\times\rtimes {\rm Gal}(E/F)$, and the main result of this article is as follows.
\begin{theorem}\label{intro_thm}
Let $F$ be a non-Archimedean local field with residue characteristic $p$, and let $D$ be a central $F$-division algebra of index $n$. Let $E/F$ be a finite Galois extension of degree $\ell$, where $\ell$ is a prime distinct from $p$ and $\ell$ does not divide $n$. Let $\pi$ be an absolutely irreducible depth-zero $\mathcal{K}$-representation $D^\times$, and let $\Pi = {\rm bc}_{E/F}(\pi)$. Let $\mathcal{L}_\Pi$ be the induced lattice in $\Pi$. Then the Tate cohomology group $\widehat{\rm H}^0({\rm Gal}(E/F),\mathcal{L}_\Pi)$ is isomorphic to the Frobenius twist $r_\ell(\pi)^{(\ell)}$. Moreover, the first Tate cohomology group $\widehat{\rm H}^1({\rm Gal}(E/F),\mathcal{L}_\Pi)$ is trivial.
\end{theorem}
Note that the mod-$\ell$ reduction $r_\ell(\pi)$ is not irreducible, in general (see \cite[Corollary 3.24]{modular_type_inner}). So, the above theorem says that the Tate cohomology group $\widehat{\rm H}^0({\rm Gal}(E/F),\mathcal{L}_\Pi)$ is not necessarily irreducible. If $\ell$ does not divide $q^n-1$, in which case the mod-$\ell$ reduction $r_\ell(\pi)$ is irreducible (see \cite[Corollary 3.22]{modular_type_inner}), then it follows from Theorem \ref{intro_thm} that $\widehat{\rm H}^0({\rm Gal}(E/F),\mathcal{L})$ is irreducible for any $D_E^\times\rtimes {\rm Gal}(E/F)$-stable lattice $\mathcal{L}$ in $\Pi$.

The paper is organized as follows. In Section $2$, we introduce various notations and discuss preliminary notions which are essential for our main theorem. In Section $3$, we discuss the notion of Tate cohomology groups and prove Theorem \ref{intro_thm}.

\section{Preliminaries}
\subsection{Notations}
Let $F$ be a non-Archimedean local field with residue characteristic $p$. Let $D$ be a central $F$-division algebra of dimension $d^2$. Let $\mathfrak{o}_F$ (resp. $\mathfrak{o}_D$) denote the ring of integers of $F$ (resp. $D$), and let $\mathfrak{p}_F$ (resp. $\mathfrak{p}_D$) be the unique maximal ideal of $\mathfrak{o}_F$ (resp. $\mathfrak{o}_D$). We denote by $D^\times$ the unit group of $D$. The group $D^\times$ has a filtration $\{U_{D}^n,\,n\geq 0\}$, where 
$$ U_{D}^0 = \mathfrak{o}_D^\times,\,\,\,\,\,\,\, U_D^{n} = 1+\mathfrak{p}_D^n.$$ 
Let ${\rm Trd}_D:D\rightarrow F$ and ${\rm Nrd}_D:D^\times\rightarrow F^\times$ be the reduced trace map and the norm map, respectively. For a finite Galois extension $K/F$, we denote by ${\rm Nr}_{E/F}$ and ${\rm Tr}_{E/F}$ the norm map and trace map, respectively. Let $D_E$ be the $E$-algebra $D\otimes_F E$ and is a central simple algebra with $D_E^{{\rm Gal}(E/F)}=D$. Note that $D_E$ is not necessarily a division algebra. If the degree of extension $[E:F]$ is coprime to $d$, then $D_E$ is a division $E$-algebra of dimension $d^2$. For any group $G$, let $G_{ab}$ denote the abelianization of $G$--which is by definition, the quotient of $G$ by the normal subgroup generated by the elements $xyx^{-1}y^{-1}$, $x,y\in G$.

We fix an algebraic closure $\overline{\mathbb{Q}}_\ell$ of $\mathbb{Q}_\ell$.  Let $\overline{\mathbb{Z}}_\ell$ be the ring of integers of $\overline{\mathbb{Q}}_\ell$, and let $\mathfrak{m}$ denote the unique maximal ideal of $\overline{\mathbb{Z}}_\ell$. The quotient $\overline{\mathbb{Z}}_\ell/\mathfrak{m}\overline{\mathbb{Z}}_\ell$ is isomorphic to $\overline{\mathbb{F}}_\ell$, an algebraic closure of the finite field of cardinality $\ell$. Let $\mathcal{K}$ be the maximal unramified extension of $\mathbb{Q}_\ell$ in $\overline{\mathbb{Q}}_\ell$, and let $\mathcal{O}_{\mathcal{K}}$ be the ring of integers of $\mathcal{K}$. Note that $\ell$ is a uniformizer of $\mathcal{K}$ with $\mathcal{O}_{\mathcal{K}}/\ell\mathcal{O}_{\mathcal{K}}\simeq \overline{\mathbb{F}}_\ell$.
For a locally profinite group $G$ and a closed subgroup $H$ of $G$, we denote by ${\rm Ind}_H^G$ and ${\rm ind}_H^G$ the smooth induction functor and the compact induction functor, respectively.  

\subsection{Integral representation and mod-$\ell$ reduction}
Let $G$ be a locally profinite group. Let $R$ denote either $\overline{\mathbb{Q}}_\ell$ or $\mathcal{K}$ or $\overline{\mathbb{F}}_\ell$. Let $(\pi,V)$ be a smooth $R$-representation of $G$, where smooth means the stabilizer of each vector in the $R$-vector space $V$ is an open subgroup of $G$. The representation $(\pi,V)$ is called $\ell$-adic (resp. $\ell$-modular) if $V$ is a vector space over $\overline{\mathbb{Q}}_\ell$ (resp. $\overline{\mathbb{F}}_\ell$). Moreover, $(\pi,V)$ is called admissible if for each compact open subgroup $K$ of $G$, the space of $K$-fixed vectors $V^K$ is finite dimensional. A smooth $\ell$-adic representation $(\pi,V)$ is said to be integral if it has finite length and there exists a $G$-stable $\overline{\mathbb{Z}}_\ell$-lattice $\mathcal{L}$ in $V$. Suppose $(\pi,V)$ is an admissible $\ell$-adic 
integral representation defined over $\mathcal{K}$, i.e., there
exists a $G$-stable $\mathcal{K}$-vector subspace $V_{\mathcal{K}}$ in $V$ such
that $V_{\mathcal{K}}\otimes_{\mathcal{K}} \overline{\mathbb{Q}}_\ell = V$, then
there exists an $\mathcal{O}_{\mathcal{K}}[G]$-stable lattice $\mathcal{L}_{\mathcal{K}}$ in $V_{\mathcal{K}}$
and $\mathcal{L}_{\mathcal{K}} \otimes_{\mathcal{O}_{\mathcal{K}}}
\overline{\mathbb{Z}}_\ell$ is a
$G$-stable $\overline {\mathbb{Z}}_\ell$-lattice in $V$. The natural action of $G$ on the $\overline{\mathbb{F}}_\ell$-vector space $\mathcal{L}\otimes_{\overline{\mathbb{Z}}_\ell}
\overline{\mathbb{F}}_\ell$ gives an $\ell$-modular representation of $G$.
Using \cite[Chapter 2, 5.11a - 5.11.b]{Vigneras_modl_book}, the representation $(\pi,\mathcal{L}\otimes_{\overline
{\mathbb{Z}}_\ell}\overline{\mathbb{F}}_\ell)$ has finite length, and its semi-simplification is independent of the choice of the $G$-invariant lattice $\mathcal{L}$ in $V$. We denote this semi-simplification by $r_\ell(\pi)$, and it is called the mod-$\ell$ reduction of $\pi$. 
\subsubsection{Frobenius twist}
Let $(\rho,V)$ be an $\ell$-modular representation of a group $G$. Let $V^{(\ell)}$ be the $\overline{\mathbb{F}}_\ell$-vector space whose additive group structure is same as that of $V$ and the scalar action on $V^{(\ell)}$ is given by 
$$ c\ast v := c^{\frac{1}{\ell}}\cdot v $$
for all $c\in \overline{\mathbb{F}}_\ell$ and $v\in V$. Let $\rho^{(\ell)}$ be the representation of $G$ on $V^{(\ell)}$, defined by $\rho^{(\ell)}(g):=\rho(g)$ for $g\in G$. The pair $(\rho^{(\ell)}, V^{(\ell)})$ is called the Frobenius twist of $(\rho,V)$.
\subsubsection{Local Langlands correspondence}
We now recall the local Langlands correspondence for the group ${\rm GL}_n(F)$. Let us introduce some notations. Let $\mathcal{W}_F$ be the Weil-group of $F$, and let $\nu$ be the absolute value on $\mathcal{W}_F$ that sends the geometric Frobenius element in $\mathcal{W}_F$ to $q^{-1}$, where $q$ is the cardinality of the residue field of $F$. An $n$-dimensional semisimple Weil-Deligne $\mathbb{C}$-representation of $\mathcal{W}_F$ is a pair $(\varphi,\mathfrak{n})$ consisting of an $n$-dimensional semisimple $\mathbb{C}$-representation $\varphi$ of $\mathcal{W}_F$ and a nilpotent endomorphism $\mathfrak{n}$ on the space of $\varphi$ such that 
$$ \varphi(x)\mathfrak{n}\varphi(x)^{-1} = \nu(x)\mathfrak{n} $$
for all $x\in\mathcal{W}_F$. Let $\mathcal{G}_\mathbb{C}(n,F)$ be the set of isomorphism classes of semisimple Weil-Deligne representations of $\mathcal{W}_F$ of dimension $n$, and let ${\rm Irr}_{\mathbb{C}}(n,F)$ denote the set of isomorphism classes of irreducible smooth $\mathbb{C}$-representations of ${\rm GL}_n(F)$. The local Langlands correspondence is a bijection
$$ {\rm Irr}_{\mathbb{C}}(n,F)\longrightarrow \mathcal{G}_{\mathbb{C}}(n,F) $$
which is uniquely characterized by the preservation of the local $L$ and $\epsilon$-factors associated to both sides. By fixing an isomorphism $\iota:\mathbb{C}\xrightarrow{\sim}\overline{\mathbb{Q}}_\ell$, we obtain an $\ell$-adic local Langlands correspondence 
$$ {\rm LLC}_F : {\rm Irr}_{\overline{\mathbb{Q}}_\ell}(n,F)\longrightarrow\mathcal{G}_{\overline{\mathbb{Q}}_\ell}(n,F), $$
which is also uniquely characterized by the preservation of local $L$ and $\epsilon$-factors.
\subsubsection{}
The local base change lifting is a map
$$ {\rm BC}_{E/F}:{\rm Irr}_{\overline{\mathbb{Q}}_\ell}(n,F)\rightarrow {\rm Irr}_{\overline{\mathbb{Q}}_\ell}(n,E) $$
which is defined by a certain character identity relation (see \cite{Arthur_Clozel_BC}). There is a characterization of local base change map in terms of ${\rm LLC}_F$. Let $\pi$ be an irreducible smooth $\ell$-adic representation of ${\rm GL}_n(F)$. An irreducible $\ell$-adic representation $\Pi$ of ${\rm GL}_n(E)$ is called the base change lifting of $\pi$ if 
$$ {\rm res}_{\mathcal{W}_E}({\rm LLC}_F(\pi)) \simeq {\rm LLC}_E(\Pi). $$
Here, the Weil group $\mathcal{W}_E$ of $E$ is viewed as a subgroup of $\mathcal{W}_F$. Note that $\Pi$ is isomorphic to $\Pi^{\sigma}$ for all $\sigma\in {\rm Gal}(E/F)$.

Let $\pi$ be an essentially square-integrable representation of ${\rm GL}_n(F)$. It is characterized as the unique irreducible quotient of the parabolically induced representation 
$$ \xi\times \xi\lvert\cdot\rvert_F\times\cdots
\times\xi\lvert\cdot\rvert_F^{r-1}, $$
for a unique integer $r$ dividing $n$ and a unique cuspidal representation $\xi$ of ${\rm GL}_{n/r}(F)$. Here, $\lvert\cdot\rvert_F$ is the normalized absolute value on $F$. The representation $\pi$ is denoted by $\delta(\xi,r)$.
Assume that $\ell$ does not divide $n$. Let $\tau$ be the base change lift of $\xi$ to ${\rm GL}_{n/r}(E)$. Note that $\tau$ is cuspidal, and the parabolically induced representation
$$ \tau\times\tau\lvert{\rm Nr}_{E/F}\rvert_F\times\cdots\times
\tau\lvert{\rm Nr}_{E/F}\rvert_F^{r-1} $$ has a unique irreducible quotient, denoted by $\delta(\tau,r)$. The representation $\delta(\tau,r)$ is the base change lift of $\delta(\xi,r)$.

\subsubsection{}
Before proceeding further, we now recall certain Weil-Deligne representations of $\mathcal{W}_F$. Let $V$ be an $n$-dimensional $\overline{\mathbb{Q}}_\ell$-vector space with basis $v_0, v_1,\dots,v_{n-1}$. Let $\mathfrak{n}$ be an endomorphism of $V$, defined on the basis elements as follows
$$ \mathfrak{n}v_i = v_{i+1},\,\,\,\, 
0\leq i< n-1,$$
$$ \mathfrak{n}v_{n-1} = 0.$$
Let $\rho$ be a semisimple smooth representation of $\mathcal{W}_F$ on $V$, defined as
$$ \rho(x) v_i = \nu(x)^{i+(1-n)/2}\,v_i,\,\,\,\, 0\leq i\leq n-1 $$
for all $x\in\mathcal{W}_F$. The pair $(\rho,\mathfrak{n})$ is an $n$-dimensional semisimple Weil-Deligne representation of $\mathcal{W}_F$, and is denoted by ${\rm Sp}_n(F)$. 

\subsection{Local Jacquet-Langlands correspondence}
In this part, we discuss Jacquet-Langlands correspondence and the base change lifting of irreducible $\ell$-adic representations of $D^\times$, where $D$ is a central $F$-division algebra of dimension $d^2$. Let $R$ be either $\mathbb{C}$ or $\overline{\mathbb{Q}}_\ell$. Let ${\rm Irr}_{R}(D^\times)$ be the set of isomorphism classes of irreducible $R$-representations of $D^\times$. Let $\mathcal{D}_R({\rm GL}_d(F))$ denote the set of isomorphism classes of essentially square-integrable $R$-representations of ${\rm GL}_d(F)$. When $R=\mathbb{C}$, the local Jacquet-Langlands correspondence (\cite{Automorphic_GL2}, \cite{Badulescu_JL}) is a bijection
$$ {\rm Irr}_{\mathbb{C}}(D^\times)\rightarrow 
\mathcal{D}_{\mathbb{C}}({\rm GL}_d(F)) $$
which is characterized by the character identity relation on the regular semisimple elements. We fix an isomorphism $\iota:\mathbb{C}\xrightarrow{\sim} \overline{\mathbb{Q}}_\ell$. Then there is an $\ell$-adic Jacquet-Langlands correspondence (see \cite{MS_JL})
$$ {\rm JL}_F : \mathcal{D}_{\overline{\mathbb{Q}}_\ell}
({\rm GL}_m(D))\rightarrow \mathcal{D}_{\overline{\mathbb{Q}}_\ell}({\rm GL}_n(F)) $$
and is independent of the choice of $\iota$.
\subsubsection{Base change}
Let $E/F$ be a Galois extension of prime degree $\ell$, where $\ell\ne p$ and $\ell$ does not divide $d$. Set $D_E = D\otimes_F E$. Then $D_E$ is a central $E$-division algebra of dimension $d^2$. Since $\ell$ does not divide $d$, the base change map ${\rm BC}_{E/F}$ takes essentially square-integrable representations of ${\rm GL}_d(F)$ to essentially square-integrable representations of ${\rm GL}_d(E)$, and we have the following commutative diagram 
$$\xymatrix{
	{\rm Irr}_{\overline{\mathbb{Q}}_\ell}(D^\times)\ar[dd]_{{\rm bc}_{E/F}}  \ar[rr]^{{\rm JL}_F}  &&  \mathcal{D}_{\overline{\mathbb{Q}}_\ell}({\rm GL}_d(F)) \ar[dd]^{{{\rm BC}_{E/F}}} \\\\
	{\rm Irr}_{\overline{\mathbb{Q}}_\ell}(D_E^\times) \ar[rr]_{{\rm JL}_E} &&  \mathcal{D}_
	{\overline{\mathbb{Q}}_\ell}({\rm GL}_d(E))
}$$
where the vertical map ${\rm bc}_{E/F}$ is defined by the composition ${\rm JL}_E^{-1}\circ {\rm BC}_{E/F}\circ {\rm JL}_F$. Let $\pi$ be an irreducible smooth $\ell$-adic representation of $D^\times$. An irreducible smooth $\ell$-adic representation $\Pi$ of $D_E^\times$ is called the {\it base change} of $\pi$ if $\Pi={\rm bc}_{E/F}(\pi)$. Since ${\rm BC}_{E/F}({\rm JL}_F(\pi)) \simeq {\rm BC}_{E/F}({\rm JL}_F(\pi))^\sigma$ for all $\sigma\in {\rm Gal}(E/F)$ and the Jacquet-Langlands correspondence preserves objects in ${\rm Irr}_{\overline{\mathbb{Q}}_\ell}(D_E^\times)$ that are ${\rm Gal}(E/F)$-equivariant, we have $\Pi\simeq \Pi^\sigma$ for all $\sigma\in{\rm Gal}(E/F)$.   
\subsection{Depth-zero representations of $D^\times$}
We now discuss the depth-zero representations of $D^\times$ and their construction using admissible tame pairs. We follow the exposition of \cite{BH_depth_zero}. An irreducible smooth $\ell$-adic representation $\pi$ of $D^\times$ is said to be of depth-zero if $\pi|_{U_D^1}$ is a trivial representation. Since the Jacquet-Langlands correspondence preserves depths in the sense of \cite{Moy_Prasad_depths} and the base change lifting ${\rm BC}_{E/F}$ takes an irreducible representation of ${\rm GL}_d(F)$ of depth $r$ to an irreducible representation of ${\rm GL}_d(E)$ of depth $er$, where $e$ is the ramification index of $E/F$, so does the map ${\rm bc}_{E/F}$. In particular, it preserves depth-zero representations of $D^\times$.
There is an explicit construction of depth-zero representations of $D^\times$ in terms of tame admissible pairs--the notion that we will discuss next.
\subsubsection{Admissible pair}\label{Adm_pair}
A tame pair over $F$ consists of a finite unramified extension $K$ of $F$ and a smooth character $\theta$ of $K^\times$ which is tamely ramified, i.e., $\theta$ is trivial on $U_K^1$. Moreover, the pair $(K/F,\theta)$ is called {\it admissible} if the Galois conjugates $\theta^\gamma$, $\gamma\in{\rm Gal}(K/F)$, are distinct. Two admissible tame pairs $(K_i/F,\theta_i)$, $i=0,1$, are said to be isomorphic if there exists an $F$-isomorphism $\alpha:K_1\rightarrow K_2$ such that $\theta_1=\theta_2\circ \alpha$.

We now recall the construction of irreducible smooth depth-zero representations of $D^\times$ using admissible tame pairs. Let $(K/F,\theta)$ be an admissible tame pair such that $[K:F]$ divides $d$. Choose an $F$-embedding $K\hookrightarrow D$ and identify $K$ as an $F$-subalgebra of $D$. Let $B$ be the $D$-centralizer of $K$. Then $B$ is a central $K$-division algebra of dimension $e=d/[K:F]$. Set $J=B^\times U_D^1$, and let $\Lambda$ be the character of $J$ defined as 
$$ \Lambda(bu) = \theta({\rm Nrd}_B(b)), $$
where ${\rm Nrd}_B:B^\times\rightarrow K^\times$ is the reduced norm map. Then the induced representation 
$$ \pi_D(\theta) := {\rm ind}_{J}^{D^\times}(\Lambda) $$
is an irreducible smooth depth-zero $\ell$-adic representation of $D^\times$. The isomorphism class of $\pi_D(\theta)$ depends on the isomorphism class of the admissible tame pair $(K/F,\theta)$. In fact, any irreducible smooth depth-zero representation of $D^\times$ is constructed in this way.

On the other hand, there is a depth-zero $\ell$-adic representation of $\mathcal{W}_F$ associated with the pair $(K/F,\theta)$. We recall the construction. Composing with the Artin reciprocity map $\mathcal{W}_K\rightarrow K^\times$, the character $\theta$ of $K^\times$ induces a character of $\mathcal{W}_K$ which is again denoted by $\theta$. Then the induced representation
$$ \sigma(\theta) := {\rm ind}_{\mathcal{W}_K}^{\mathcal{W}_F}(\theta) $$
is irreducible and of depth-zero. Its isomorphism class depends only on the isomorphism class of the pair $(K/F,\theta)$. In fact, any irreducible smooth depth-zero $\ell$-adic representation of $\mathcal{W}_F$ is obtained in this way.

\subsubsection{Base change of admissible pairs}
Let $EK$ be the field $E\otimes_F K$. Let $\sigma$ be a generator of the Galois group ${\rm Gal}(E/F)$. Note that ${\rm Gal}(E/F)$ acts on both $D_E$ and $EK$ with the fixed point subgroups $D$ and $K$, respectively. Let $B_E$ be the centralizer of $EK$ in $D_E$. Note that $B_E = B\otimes_K EK$ and is stable under the action of ${\rm Gal}(E/F)$ with $B_E^\sigma= D\cap B_E =B$. Let ${\rm Nrd}_{B_E}:B_E^\times\rightarrow (EK)^\times$ be the reduced norm map. We identify the generator of ${\rm Gal}(EK/K)$ with the generator $\sigma$ of ${\rm Gal}(E/F)$ via the isomorphism 
$$ {\rm Gal}(EK/K)\rightarrow {\rm Gal}(E/F) $$
$$ \tau\longmapsto \tau|_{E}. $$
Let ${\rm Cor}_{E/F}:(B_E^\times)_{ab}\rightarrow B^\times_{ab}$ be the corestriction map defined by 
$$ {\rm Cor}_{E/F}(b) :=\prod_{\sigma\in {\rm Gal}(E/F)}\sigma(b). $$
Then we have the following relation ${\rm Nr}_{EK/K}\circ {\rm Nrd}_{B_E} = {\rm Nrd}_B\circ {\rm Cor}_{E/F}$, where ${\rm Nr}_{EK/K}$ denotes the norm map associated to the field extension $EK/K$.
Let $\widetilde{\theta}$ be the character of $(EK)^\times$, given by the composition $\theta\circ{\rm Nr}_{EK/K}$.
\begin{lemma}
The pair $(EK/E,\widetilde{\theta})$ is an admissible tame pair.
\end{lemma}
\begin{proof}
Since $K/F$ is unramified, the extension $EK/E$ is unramified of degree dividing $[K:F]$. Note that ${\rm Nr}_{EK/K}(U_{EK}^1)\subseteq U_K^1$, which implies that $\widetilde{\theta}$ is trivial on $U_{EK}^1$ as $\theta$ is trivial on $U_K^1$. Thus, we get that $(EK/E,\widetilde{\theta})$ is a tame pair. Moreover, the conjugates $\widetilde{\theta}^\gamma$, $\gamma\in {\rm Gal}(EK/E)$, are disjoint. Hence, $(EK/E,\widetilde{\theta})$ is an admissible tame pair..
\end{proof}
We set $J_E = B_E^\times U_{D_E}^1$. Note that $J_E$ is stable under the action of ${\rm Gal}(E/F)$ and $J_E^\sigma = J$. Define the character $\Lambda_E$ of $J_E$ as follows 
$$ \Lambda_E(bu) := \widetilde{\theta}({\rm Nrd}_{B_E}(b)) $$
for $b\in B_E^\times$ and $u\in U_{D_E}^1$. Then $\pi_{D_E}(\widetilde{\theta}):={\rm ind}_{J_E}^{D_E^\times}(\Lambda_E)$ is an irreducible depth-zero $\ell$-adic representation of $D_E^\times$. The next lemma shows that $\pi_{D_E}(\widetilde{\theta})$ is the base change lifting of $\pi_D(\theta)$.
\begin{lemma}
We have 
$\pi_{D_E}(\widetilde{\theta}) = {\rm bc}_{E/F}(\pi_D(\theta))$.
\end{lemma}
\begin{proof}
We prove the lemma by considering the corresponding Galois representations. Recall that
$$ ({\rm LLC}_F\circ {\rm JL}_F)(\pi_D(\theta)) = \sigma(\eta_K^{e(f-1)}\theta)\otimes {\rm Sp}_e(F). $$
and 
$$ ({\rm LLC}_E\circ {\rm JL}_E)(\pi_{D_E}(\widetilde{\theta})) = \sigma(\eta_{EK}^{e(f-1)}\widetilde{\theta})\otimes{\rm Sp}_e(E), $$
where $\eta_K$ is the quadratic character of $K^\times$ and $\eta_{EK} = \eta_K\circ {\rm Nr}_{EK/K}$. We have to show that 
$$ \sigma(\eta_K^{e(f-1)}\theta)\otimes {\rm Sp}_e(F) |_{\mathcal{W}_E} = \sigma(\eta_{EK}^{e(f-1)}\widetilde{\theta})\otimes{\rm Sp}_e(E). $$
Note that ${\rm Sp}_e(F)|_{\mathcal{W}_E} ={\rm Sp}_e(E)$. Therefore, it remains to show that 
$$ \sigma(\eta_{K}^{e(f-1)}\theta)|_{\mathcal{W}_E} =\sigma(\eta_{EK}^{e(f-1)}\widetilde{\theta}). $$
By definition, we have $\sigma(\eta_K^{e(f-1)}\theta) = \eta_K^{e(f-1)}\otimes{\rm Ind}_{\mathcal{W}_K}^{\mathcal{W}_F}(\theta)$ and $\sigma(\eta_{EK}^{e(f-1)}\theta) = \eta_{EK}^{e(f-1)}\otimes{\rm Ind}_{\mathcal{W}_{EK}}^{\mathcal{W}_E}(\widetilde{\theta})$. Also note that $\eta_K|_{\mathcal{W}_{EK}} = \eta_{EK}$. By Mackey decomposition, we have
$$ {\rm Ind}_{\mathcal{W}_{K}}^{\mathcal{W}_F}(\theta)|_{\mathcal{W}_E} = \bigoplus_{\gamma\in\mathcal{W}_K\backslash \mathcal{W}_F/\mathcal{W}_E}{\rm Ind}_{\mathcal{W}_E\cap \mathcal{W}_K^\gamma}^{\mathcal{W}_E}(\theta^{(\gamma)}|_{\mathcal{W}_E\cap
\mathcal{W}_K^\gamma}). $$
Note that the cardinality of the space of double cosets $\mathcal{W}_K\backslash \mathcal{W}_F/\mathcal{W}_E$ divides both $f$ and $\ell$, and hence equal to $1$. Since $\mathcal{W}_E\cap\mathcal{W}_K=\mathcal{W}_{EK}$ and $\theta|_{\mathcal{W}_{EK}} =\widetilde{\theta}$, we get
$$ {\rm Ind}_{\mathcal{W}_{K}}^{\mathcal{W}_F}(\theta)|_{\mathcal{W}_E} = {\rm Ind}_{\mathcal{W}_{EK}}^{\mathcal{W}_E}(\widetilde{\theta}) $$
This proves the lemma.
\end{proof}

\subsubsection{Rationality structure}
Let $\mathcal{K}$ denote the maximal unramified extension of $\mathbb{Q}_\ell$ in $\overline{\mathbb{Q}}_\ell$. Let $\pi_{\mathcal{K}}$ be an absolutely irreducible depth-zero $\mathcal{K}$-representation of $D^\times$. Using the same construction as in Subsection (\ref{Adm_pair}), we get that $\pi_{\mathcal{K}}=\pi_D(\theta)$ for an admissible tame pair $(K/F,\theta)$. Recall that $\pi_{\mathcal{K}}$ is the compactly induced representation ${\rm ind}_{J}^{D^\times}(\Lambda)$. Then $\pi=\pi_{\mathcal{K}}\otimes_{\mathcal{K}}
\overline{\mathbb{Q}}_\ell$ is an irreducible depth-zero $\ell$-adic representation of $D^\times$. Let $\Pi = {\rm bc}_{E/F}(\pi)$. Note that $\Pi = {\rm ind}_{J_E}^{D_E^\times}(\Lambda_E)$, where $\Lambda_E$ is a smooth character of $J_E$ induced by the admissible tame pair $(EK/E,\widetilde{\theta})$, where $\widetilde{\theta}=\theta\circ {\rm Nr}_{EK/K}$. Let $\Pi_{\mathcal{K}}$ be the space of $\mathcal{K}$-valued functions in $\Pi$. Then $\Pi_{\mathcal{K}}$ is an absolutely irreducible depth-zero $\mathcal{K}$-representation of $D_E^\times$, and it is a $\mathcal{K}$-rational structure of the $\ell$-adic representation $\Pi$.

\section{Tate cohomology of depth-zero representations of $D^\times$}
In this section, we begin by recalling the notion of Tate cohomology groups. Then we prove our main theorem (Theorem \ref{intro_thm}).
\subsection{Tate cohomology}
Let $\Gamma$ be a finite cyclic group of prime order $\ell$. Let $\sigma$ be a generator of $\Gamma$, and set ${\rm Nr}={\rm id}+\sigma+\cdots+\sigma^{\ell-1}$. Let $M$ be an $\mathcal{O}_{\mathcal{K}}[\Gamma]$-module. Then the Tate cohomology groups of $M$ under the action of $\Gamma$, are defined as follows.
$$ \widehat{\rm H}^0(\sigma,M) := \dfrac{{\rm Ker}({\rm id}-\sigma)}{{\rm Im}({\rm Nr})}, \,\,\, \widehat{\rm H}^1(\sigma,M) := \dfrac{{\rm Ker}({\rm Nr})}{{\rm Im}({\rm id}-\sigma)}. $$
Each Tate cohomology group $\widehat{\rm H}^i(\sigma,M)$ is a vector space over $\overline{\mathbb{F}}_\ell$. If $G$ is a group with $\Gamma\subseteq {\rm Aut}(G)$ an automorphism group of order $\ell$, and $M$ is an $\mathcal{O}_{\mathcal{K}}[G\rtimes \Gamma]$-module, then each $\widehat{\rm H}^i(\sigma,M)$ carries an action of the fixed-point subgroup $G^\sigma$, giving two $\ell$-modular representations of $G^\sigma$. We now recall a result due to Treumann--Venkatesh (\cite[Proposition 3.3.1]{TV_paper}), which gives the compatibility of Tate cohomology groups with compact induction. This result is crucial in deducing our main theorem. 
\begin{proposition}\label{TV_isom}
Let $G$ be a group and $\sigma$ be an automorphism of $G$ of prime order $\ell$. Let $H$ be a subgroup of $G$ such that $\sigma(H) = H$. Let $V$ be an $\mathcal{O}_{\mathcal{K}}[H\rtimes \langle\sigma\rangle]$-module. Then the restriction map 
$$ \big({\rm ind}_H^G(v)\big)^\sigma \rightarrow {\rm ind}_{H^\sigma}^{G^\sigma}(\widehat{\rm H}^i(\sigma, V)) $$
$$ \varphi\mapsto \varphi|_{G^\sigma} $$
induces an isomorphism of $G^\sigma$-representations
$$ \widehat{\rm H}^i\big(\sigma, {\rm ind}_H^G(V)\big) \simeq 
{\rm ind}_{H^\sigma}^{G^\sigma}(\widehat{\rm H}^i(\sigma,V)). $$ 
\end{proposition}
\subsection{}
Let $\pi$ be an absolutely irreducible, depth-zero, integral, $\mathcal{K}$-representation of $D^\times$. Then $\pi = \pi_D(\theta)$ for some admissible tame pair $(K/F,\theta)$, where the character $\theta$ of $K^\times$ takes values in $\mathcal{O}_{\mathcal{K}}$. Recall that the base change lifting ${\rm bc}_{E/F}(\pi)$ is of the form $\pi_{D_E}(\widetilde{\theta})$ for the admissible tame pair $(EK/E,\widetilde{\theta})$, where $\widetilde{\theta}$ is the composition $\theta\circ {\rm Nr}_{EK/K}$. We denote by $\Pi$ the $D_E^\times$-representation $\pi_{D_E}(\widetilde{\theta})$. Note that $\Pi\simeq \Pi^\sigma$ for all $\sigma\in {\rm Gal}(E/F)$.
\begin{lemma}
The representation $\Pi$ extends uniquely to a representation of $D_E^\times\rtimes {\rm Gal}(E/F)$.
\end{lemma}
\begin{proof}
Recall that $\Pi = {\rm ind}_{J_E}^{D_E^\times}(\Lambda_E)$, where $J_E = B_E^\times U_{D_E}^1$ and $\Lambda_E$ is the character of $J_E$ given by 
$$ \Lambda_E(bu) = \widetilde{\theta}({\rm Nrd}_{B_E}(b)). $$
Note that \begin{align*}
\Lambda_E^\sigma(bu) 
&= \Lambda_E(\sigma(b)\sigma(u))\\
&=\widetilde{\theta}({\rm Nrd}_{B_E}(\sigma(b)))\\
&=\theta\big(({\rm Nr}_{EK/K}\circ {\rm Nrd}_{B_E})(\sigma(b))\big)
\end{align*}
Using the relation ${\rm Nr}_{EK/K}\circ{\rm Nrd}_{B_E} = {\rm Nrd}_B\circ {\rm Cor}_{E/F}$, we get
\begin{align*}
\Lambda_E^\sigma(bu)
&=\theta\big(({\rm Nrd}_B\circ {\rm Cor}_{E/F})(\sigma(b))\big)\\
&=\theta\big(({\rm Nrd}_B\circ {\rm Cor}_{E/F})(b)\big)\\
&=\widetilde{\theta}({\rm Nrd}_{B_E}(b))\\
&=\Lambda_E(bu).
\end{align*}
Thus, the character $\Lambda_E$ is ${\rm Gal}(E/F)$-equivariant and hence ${\rm Gal}(E/F)$ acts trivially on $\Lambda_E$. This defines an action of ${\rm Gal}(E/F)$ on $\Pi$, given by
$$ (\sigma\cdot f)(g) = f(\sigma^{-1}(g)) $$
for all $g\in D_E^\times$. 

Suppose that there exist two actions of ${\rm Gal}(E/F)$ on the $\mathcal{K}[D_E^\times]$-module $\Pi$ given by $T_\sigma$ and $T_\sigma'$. Then 
$T_\sigma = cT_\sigma'$ for some $c\in \mathcal{K}$. Since $T_\sigma^\ell = T_\sigma'^\ell = {\rm id}$, we get that $c^\ell =1$, which implies that $c=1$ as $\mathcal{K}$ is the maximal unramified extension of $\mathbb{Q}_\ell$. This completes the proof. 
\end{proof}
\subsubsection{}
Consider the long exact sequence of non-abelian cohomology
$$ 0\longrightarrow J\longrightarrow D^\times\longrightarrow (J_E\backslash D_E^\times)^\sigma\longrightarrow {\rm H}^1(\sigma, J_E). $$
We also have the following exact sequence
\begin{equation}\label{exact_seq_1}
{\rm H}^1(\sigma,U_{D_E}^1)\longrightarrow {\rm H}^1(\sigma,J_E)\longrightarrow {\rm H}^1(\sigma,B_E^\times)
\end{equation}
Since $U_{D_E}^1$ is a pro-$p$-group and $\ell\ne p$, the cohomology set ${\rm H}^1(\sigma, U_{D_E}^1)$ is trivial. Next, we recall a general result which shows that ${\rm H}^1(\sigma,B_E^\times)$ is trivial. 
\begin{lemma}\label{cohom_triv}
Let $M$ be a finite dimensional unitary $F$-algebra. Let $M_E = M\otimes_F E$. Then ${\rm H}^1(\sigma,M_E^\times)$ is trivial.
\end{lemma}
\begin{proof}
Let $\Gamma = {\rm Gal}(E/F)$. Let $f:\Gamma\rightarrow M_E^\times$ be a $1$-cocycle (i.e., $f(\sigma\tau)=f(\sigma)\sigma(f(\tau))$) and $c\in M_E$. Consider the Poincar\'e series
$$ b= \sum_{\sigma\in\Gamma}f(\sigma)\sigma(c). $$
Note that $\sigma(b) = f(\sigma)^{-1} b$. Since $F$ is infinite, it follows from the algebraic independence of automorphisms \cite[Chapter 5, Section 10, Theorem 4]{bourbaki2003algebra} that one can choose $c$ in such a way that $b$ is invertible in $M_E$. This shows that $f$ is a coboundary. Hence, the lemma.
\end{proof}
Since $B$ is a finite-dimensional unital $K$-algebra, the cohomology set ${\rm H}^1(\sigma, B_E^\times)$ is also trivial (by Lemma \ref{cohom_triv}). Then it follows from the exact sequence (\ref{exact_seq_1}) that ${\rm H}^1(\sigma,J_E)$ is trivial. Hence, we have the following equality of sets
\begin{equation}\label{coset_equal}
(J_E\backslash D_E^\times)^\sigma = 
J\backslash D^\times.
\end{equation}
\subsubsection{}
Let $\mathcal{L}_{\Pi}$ denote the space of $\mathcal{O}_{\mathcal{K}}$-valued functions in $\Pi$. Then $\mathcal{L}_{\Pi}$ is an $\mathcal{O}_{\mathcal{K}}[D_E^\times]$-stable lattice in $\Pi$ (\cite[Proposition II.3]{Vigneras_Highest_Whittaker_model}). Also note that $\mathcal{L}_{\Pi}$ is stable under the action of ${\rm Gal}(E/F)$. Then we have the following result.
\begin{theorem}\label{main_thm}
Let $\pi$ be an absolutely irreducible depth-zero integral $\mathcal{K}$-representation of $D^\times$. Let $\Pi = {\rm bc}_{E/F}(\pi)$. Then The Tate cohomology group $\widehat{\rm H}^0(\sigma, \mathcal{L}_{\Pi})$ is isomorphic to the Frobenius twist $r_\ell(\pi)^{(\ell)}$ as a representation of $D^\times$. Moreover, the first Tate cohomology group $\widehat{\rm H}^1(\sigma,\mathcal{L}_{\Pi})$ is trivial.
\end{theorem}
\begin{proof}
Recall that $\pi = {\rm ind}_{J}^{D^\times}(\Lambda)$ and $\Pi = {\rm ind}_{J_E}^{D_E^\times}(\Lambda_E)$, where $\Lambda$ and $\Lambda_E$ are the characters of $J$ and $J_E$, defined as 
$$ \Lambda(bu) = \theta({\rm Nrd}_B(b)),\,\,\, b\in B^\times, u\in U_D^1 $$
and 
$$ \Lambda_E(b'u') = \widetilde{\theta}({\rm Nrd}_{B_E}(b')),\,\,\,\, b'\in B_E^\times, u'\in U_{D_E}^1, $$
respectively. Note that $\Lambda_E^\sigma = \Lambda_E$.
Applying Treumann--Venkatesh isomorphism, we get the following $D^\times$-isomorphism 
\begin{equation}\label{TV_isom_1}
\widehat{\rm H}^i(\sigma, \mathcal{L}_{\Pi}) \simeq {\rm ind}_{J}^{D^\times}(\widehat{\rm H}^i(\sigma,\Lambda_E)) 
\end{equation}
for each $i\in \{0,1\}$. Note that
$\widehat{\rm H}^0(\sigma,\Lambda_E) = \overline{\Lambda}_E|_{J} = \overline{\Lambda}^\ell$ and $\widehat{\rm H}^1(\sigma,\Lambda_E) = 0$, where $\overline{\Lambda}$ and $\overline{\Lambda}_E$ are the mod-$\ell$ reductions of $\Lambda$ and $\Lambda_E$, respectively. Then it follows from (\ref{TV_isom_1}) that 
$$ \widehat{\rm H}^0(\sigma,\mathcal{L}_{\Pi})\simeq r_\ell(\pi)^{(\ell)} $$
as representations of $D^\times$, and the first Tate cohomology group $\widehat{\rm H}^1(\sigma,\mathcal{L}_{\Pi})$ is trivial. Hence, the theorem.
\end{proof}
As a corollary, we have
\begin{corollary}
Let $\pi$ be an absolutely irreducible depth-zero integral $\mathcal{K}$-representation of $D^\times$. Let $\Pi = {\rm bc}_{E/F}(\pi)$. Assume further that $\ell$ does not divide $q^d-1$. Then, for any $\mathcal{O}_{\mathcal{K}}[D_E^\times\rtimes {\rm Gal}(E/F)]$-stable lattice $\mathcal{L}$ in $\Pi$, the zeroth Tate cohomology group $\widehat{\rm H}^0(\sigma, \mathcal{L})$ is irreducible, and the first Tate cohomology group $\widehat{\rm H}^1(\sigma,\mathcal{L})$ is trivial.
\end{corollary}
\begin{proof}
Let $\mathcal{L}$ be an $\mathcal{O}_{\mathcal{K}}[D_E^\times\rtimes{\rm Gal}(E/F)]$-stable lattice in $\Pi$. Since $\ell\nmid q^d-1$, the mod-$\ell$ reduction $r_\ell(\Pi)$ is irreducible, and therefore $\mathcal{L}$ is homothetic to the induced lattice $\mathcal{L}_{\Pi}$. Since $\widehat{\rm H}^1(\sigma,\mathcal{L}_{\Pi})$ is trivial (by Theorem \ref{main_thm}), we get that $\widehat{\rm H}^1(\sigma,\mathcal{L})$ is trivial.

Note that $\widehat{\rm H}^0(\sigma,\mathcal{L}_{\Pi})\simeq r_\ell(\pi)^{(\ell)}$ (by Theorem \ref{main_thm}). Since $\ell\nmid q^d-1$, the mod-$\ell$ reduction $r_\ell(\pi)$ is also irreducible, and hence $\widehat{\rm H}^0(\sigma,\mathcal{L}_\Pi)$ is irreducible. Consider the short exact sequence of $\mathcal{O}_{\mathcal{K}}[{\rm Gal}(E/F)]$-modules
$$ 0\longrightarrow\mathcal{L}\xrightarrow{{\rm mod}-\ell}\mathcal{L}\longrightarrow\mathcal{L}/\ell\mathcal{L}\longrightarrow 0$$
Since $\widehat{\rm H}^1(\sigma,\mathcal{L}) =0$, it follows from the long exact sequence of Tate cohomology corresponding to the above short exact sequence that
$$ \widehat{\rm H}^0(\sigma,\mathcal{L})\simeq \widehat{\rm H}^0(\sigma,\mathcal{L}/\ell\mathcal{L}).$$
Since $\mathcal{L}/\ell\mathcal{L}$ is irreducible, the Tate cohomology $\widehat{\rm H}^0(\sigma,\mathcal{L}/\ell\mathcal{L})$ is independent of the choice of $\mathcal{L}$, and so is $\widehat{\rm H}^0(\sigma,\mathcal{L})$. Thus, $\widehat{\rm H}^0(\sigma,\mathcal{L})$ is irreducible.
\end{proof}

\bibliographystyle{amsalpha}
\bibliography{Tate_inner}
\vspace{0.2 cm}
Sabyasachi Dhar,\\
Department of Mathematics, Indian Institute of Technology Bombay, Mumbai, Maharashtra-400076, India.\\
\texttt{mathsabya93@gmail.com},
\texttt{sabya@math.iitb.ac.in}

\end{document}